\documentclass[12pt]{amsart}
\usepackage{amscd}      
\usepackage{amssymb}
\usepackage{amsmath, amsthm, graphics}
\usepackage{xypic}      
\LaTeXdiagrams          
\usepackage[all]{xy}
\xyoption{2cell} \UseAllTwocells \xyoption{frame} \CompileMatrices
\allowdisplaybreaks[3]

\usepackage{latexsym}
\usepackage{epsfig}
\usepackage{amsfonts}
\usepackage{enumerate}
\usepackage{times}

\newtheorem{theorem}[subsection]{Theorem}

\newtheorem{lemma}[subsection]{Lemma}

\newtheorem{proposition}[subsection]{Proposition}
\theoremstyle{definition}

\theoremstyle{remark}

\theoremstyle{remark}

\newtheorem{remark}[subsection]{Remark}

\numberwithin{equation}{section}

\newcommand{\bpsi}{\bar{\psi}}
\newcommand{\Mbar}{\overline{\M}}

\newcommand{\com}{\mathbb{C}}

\newcommand{\X}{\mathcal{X}}
\newcommand{\Y}{\mathcal{Y}}

\newcommand{\M}{\mathcal{M}}
\newcommand{\C}{\mathcal{C}}

\newcommand{\bc}{\mathbf{c}}

\newcommand{\sO}{\mathcal{O}}

\newcommand{\bE}{\mathbb{E}}

\newcommand{\Mwtilde}{\widetilde{\M}}

\def\<{\left\langle}
\def\>{\right\rangle}

\begin{document}

\title{On degree zero elliptic orbifold Gromov-Witten invariants}
\author{Hsian-Hua Tseng}
\address{Department of Mathematics\\ Ohio State University\\ 100 Math Tower, 231 West 18th Ave.\\Columbus\\ OH 43210\\ USA}
\email{hhtseng@math.ohio-state.edu}

\date{\today}

\begin{abstract}
We compute, by two methods, the genus one degree zero orbifold Gromov-Witten invariants with non-stacky insertions which are exceptional cases of the dilaton and divisor equations. One method involves a detailed analysis of the relevant moduli spaces. The other method, valid in the presence of torus actions with isolated fixed points, is virtual localization. Our computations verify the conjectural evaluations of these invariants. Some genus one twisted orbifold Gromov-Witten invariants are also computed.
\end{abstract}

\maketitle

\section{introduction}
We work over $\mathbb{C}$. Let $\X$ be a proper smooth Deligne-Mumford stack with projective coarse moduli space $X$. Let $$\<\tau_{a_1}(\gamma_1)...\tau_{a_n}(\gamma_n)\>_{g,n,d}^\X$$ denote Gromov-Witten invariants of $\X$. We refer to \cite{AGV} for the notation used here, as well as the algebraic definition of these invariants. Gromov-Witten invariants of $\X$ satisfy the so-called {\em dilaton equation} and {\em divisor equation}. Let $1\in H^0(\X)$ be the Poincar\'e dual of the fundamental class, and $\gamma\in H^2(\X)$. For classes $\gamma_1,...,\gamma_n$ in the orbifold cohomology of $\X$ and non-negative integers $a_1,...,a_n$, we have
\begin{equation}\label{DE_and_DE}
\begin{split}
\<\tau_{a_1}(\gamma_1)...\tau_{a_n}(\gamma_n)\tau_1(1)\>^\X_{g,n+1,d}=&(2g-2+n)\<\tau_{a_1}(\gamma_1)...\tau_{a_n}(\gamma_n)\>^\X_{g,n,d}\\
\<\tau_{a_1}(\gamma_1)...\tau_{a_n}(\gamma_n)\tau_0(\gamma)\>^\X_{g,n+1,d}=&\left(\int_d\gamma \right)\<\tau_{a_1}(\gamma_1)...\tau_{a_n}(\gamma_n)\>^\X_{g,n,d}\\
&+\sum_{i=1}^n\<\tau_{a_1}(\gamma_1)...\tau_{a_{i-1}}(\gamma_{i-1})\tau_{a_i-1}(\gamma_i\cup\gamma)\tau_{a_{i+1}}(\gamma_{i+1})...\tau_{a_n}(\gamma_n)\tau_1(1)\>^\X_{g,n,d}.
\end{split}
\end{equation}
A detailed discussion of these equations can be found in \cite{AGV}, Section 8.3.

Let $\Mbar_{g,n+1}(\X,d)'$ be the moduli stack\footnote{In \cite{AGV} the symbol $\mathcal{K}$ is used in place of $\Mbar$ here.} of $(n+1)$-pointed genus-$g$ degree-$d$ stable maps to $\X$ such that the last marked point is non-stacky. The proofs of the dilaton and divisor equations (\ref{DE_and_DE}) are based on the interpretation of the moduli stack $\Mbar_{g,n+1}(\X,d)'$ as the universal family over the moduli stack $\Mbar_{g,n}(\X,d)$ of  $n$-pointed genus-$g$ degree-$d$ stable maps to $\X$. More precisely, the morphism $$\Mbar_{g,n+1}(\X,d)'\to \Mbar_{g,n}(\X,d)$$ defined by forgetting the last (non-stacky) marked point gives the universal family over $\Mbar_{g,n}(\X,d)$. In genus $g=1$, such a description is invalid when $d=0$ and $n=0$ because stability is violated. Let $$\Mbar_{1,1}(\X,0)'$$ denote the moduli stack parametrizing genus $1$ degree $0$ stable maps to $\X$ with one {\em non-stacky} marked point. The virtual dimension of $\Mbar_{1,1}(\X,0)'$ is computed by Riemann-Roch:
$$\text{vdim}_\com\Mbar_{1,1}(\X,0)'=1.$$
Let $[\Mbar_{1,1}(\X,0)']^{vir}$ denote the associated virtual fundamental class. The following genus $1$ orbifold Gromov-Witten invariants arising from this moduli space are of basic interest:
\begin{equation}\label{genus1-inv}
\begin{split}
&\<\tau_1(1)\>_{1,1,0}'^{\X}:=\int_{[\Mbar_{1,1}(\X,0)']^{vir}}\bpsi_1,\\
&\<\tau_0(D)\>_{1,1,0}'^{\X}:=\int_{[\Mbar_{1,1}(\X,0)']^{vir}}ev_1^*D, \,\, D\in H^2(\X, \com).
\end{split}
\end{equation}
Here $ev_1: \Mbar_{1,1}(\X,0)'\to \X$ is the evaluation map, and $\bpsi_1$ is the descendant class. Since $\Mbar_{1,1}(\X,0)'$ cannot be interpreted as the universal family, invariants (\ref{genus1-inv}) cannot be calculated by dilaton and divisor equations. 

The main result of this paper is the evaluations of (\ref{genus1-inv}):
\begin{theorem}
\begin{equation}\label{dilaton}
\<\tau_1(1)\>_{1,1,0}'^{\X}=\frac{1}{24}\int_{II\X}c_{top}(T_{II\X});
\end{equation}

\begin{equation}\label{divisor}
\<\tau_0(D)\>_{1,1,0}'^{\X}=-\frac{1}{24}\int_{II\X}\pi_\X^*D\cup c_{top-1}(T_{II\X}).
\end{equation}
\end{theorem}

Here $II\X$ is the {\em double inertia stack} associated to $\X$. By definition $II\X$ is the stack of triples $(x, g_1, g_2)$ where $x\in Ob(\X)$ is an object of $\X$ and $g_1,g_2\in Aut(x)$ are two elements of the automorphism group of $x$ such that $g_1g_2=g_2g_1$. The map $$\pi_\X: II\X\to \X$$ is the natural projection defined by $(x, g_1,g_2)\mapsto x$. 

In \cite{jt}, equations\footnote{Note that different notations were used in \cite{jt}.} (\ref{dilaton})-(\ref{divisor}) were conjectured based on a heuristic analysis of $\Mbar_{1,1}(\X,0)'$. They play an important in the formulation of Virasoro constraints for orbifold Gromov-Witten theory, see \cite{jt}. They are also the first calculations of genus $1$ orbifold Gromov-Witten invariants that are valid in full generality.\footnote{It is worth noting that by the work of \cite{jk}, (\ref{dilaton}) holds true for $\X=BG$ the classifying stack of a finite group $G$, and (\ref{divisor}) is vacuous in this case.}

 For a smooth complex projective variety $X$, the calculations of analogous invariants $\<\tau_1(1)\>_{1,1,0}^X$ and $\<\tau_0(D)\>_{1,1,0}^X$ (see e.g. \cite{ge}) follows from an explicit description of the moduli stack $\Mbar_{1,1}(X,0)$ of $1$-pointed genus $1$ degree $0$ stable maps to $X$ as a product $\Mbar_{1,1}\times X$ of the moduli stack $\Mbar_{1,1}$ of $1$-pointed  genus $1$ stable curves and $X$. One can then deduce from this description an explicit formula for the virtual fundamental class needed for the calculations.

However for a Deligne-Mumford stack $\X$ the moduli stack $\Mbar_{1,1}(\X,0)'$ no longer admits such a simple description. As a result the calculations of $\<\tau_1(1)\>_{1,1,0}'^{\X}$ and $\<\tau_0(D)\>_{1,1,0}'^{\X}$ are harder. In this paper we study (\ref{dilaton})-(\ref{divisor}) by two methods. The first method is based on an analysis of the moduli stack $\Mbar_{1,1}(\X, 0)'$. The main idea is to compare $\Mbar_{1,1}(\X, 0)'$ with the product $\Mbar_{1,1}\times II\X$, where $\Mbar_{1,1}$ is the moduli stack of $1$-pointed genus $1$ stable curves. We show that there are two natural maps $\pi_1, \pi_2$ fitting into a diagram
\begin{equation*}
\xymatrix{
\Mbar_{1,1}(\X, 0)'\ar[dr]_{\pi_1}& {} & {\Mbar_{1,1}\times II\X}\ar[ld]^{\pi_2}\\
{} & {\Mbar_{1,1}\times \X_2.} &{}
}
\end{equation*}
Here $\X_2$ denotes the stack of pairs $(x, H)$ where $x\in Ob(\X)$ is an object of $\X$ and $H\subset Aut(x)$ is an abelian subgroup with at most two generators. We show that $\pi_1$ and $\pi_2$ are \'etale of the same degree, and we explicitly describe the obstruction bundle of $\Mbar_{1,1}(\X, 0)'$ as the pull-back by $\pi_1$ of a vector bundle on $\Mbar_{1,1}\times \X_2$. These are the main ingredients used to establish (\ref{dilaton})-(\ref{divisor}). We first establish these properties for the moduli stack $\Mbar_{1,1}(\X, 0)'$ in case when $\X=[M/G]$ is a global quotient by a finite group; see Section \ref{section:global_quotient_case}. The theory of admissible covers \cite{acv} plays an important role in this analysis. Then we use the fact that every Deligne-Mumford stack is locally a quotient by a finite group to extend our analysis to general $\X$; see Section \ref{section:general_case}. 

The second method, valid for stacks $\X$ admitting an algebraic torus action with isolated fixed points, is to compute (\ref{dilaton})-(\ref{divisor}) by the virtual localization formula \cite{gp}; see Section \ref{section:localization_approach}. 

Our analysis also allows use to do some related calculations of some genus $1$ {\em twisted} orbifold Gromov-Witten invariants \cite{t}. This is explained in Section \ref{section:twisted_invariants}.

\subsection*{Acknowledgment}
The author thanks D. Abramovich, K. Behrend, T. Graber, M. Olsson, and A. Vistoli for valuable discussions, and the referee for suggestions and pointing out a number of errors in the previous version of this paper. Part of this work was done during a visit to Mathematical Sciences Research Institute in Spring 2009. The author is grateful for their hospitality and support. In addition, the author is supported in part by NSF grant DMS-0757722.

\section{Global quotient case}\label{section:global_quotient_case}
In this Section we consider the case when the target $\X=[M/G]$ is a global qoutient of a smooth variety $M$ by a finite group $G$. 

\subsection{Moduli stacks}
Let $\X=[M/G]$ be a global qoutient of a smooth (not necessarily proper) scheme $M$ by a finite group $G$. Let $S$ be a scheme, and let $$[f: (\C/S, p)\to \X]\in \Mbar_{1,1}(\X, 0)'$$ be an $S$-valued point of $\Mbar_{1,1}(\X, 0)'$. Here $p: S\to \C$ is the marked section, whose image is contained in the non-singular non-stacky locus of $\C$. Let $(C, \bar{p})$ denote the coarse curve of $(\C, p)$.

Let $M\to \X=[M/G]$ be the natural atlas for $\X$. Set $$D:= \C\times_{f, [M/G]} M.$$ We obtain the following diagram
$$\begin{CD}
D@> >> M\\
@V{}VV @V{}VV\\
(\C, p)@>f >>[M/G]\\
@V{}VV\\
(C, \bar{p}).
\end{CD}$$
It is known that the composite $$D\to \C\to C$$ is an {\em admissible $G$-cover} in the sense of \cite{acv}, and the map $D\to M$ is $G$-equivariant and of degree $0$. Since the marked section $p: S\to \C$ is non-stacky, the admissible $G$-cover $D\to C$ is unramified over $\bar{p}$. Given the admissible $G$-cover $D\to C$, we may recover the twisted curve $\C$ as the stack quotient $\C=[D/G]$. Moreover, the data $f: (\C, p)\to [M/G]$ is equivalent to the data of an admissible $G$-cover $D\to (C, \bar{p})$ unramified over $\bar{p}$ and a $G$-equivariant map $D\to M$. 

Let $\Mwtilde_{1,1}([M/G],0)'$ be the moduli space whose $S$-valued points are diagrams 
\begin{equation}\label{object_in_Mwtilde}
\begin{CD}
(D, p')@> >>M\\
@V{}VV\\
(C/S, \bar{p}),
\end{CD}
\end{equation}
such that
\begin{enumerate}
\item
$(C, \bar{p})$ is a connected $1$-pointed genus $1$ stable curve over $S$;
\item
$D\to C$ is an $S$-family of admissible $G$-covers unramified over $\bar{p}(S)$;
\item
$D\to M$ is $G$-equivariant and of degree $0$;
\item
$p': S\to D$ is a section of $D/S$ such that the composite $S\overset{p'}{\longrightarrow}D\to C$ is the map $\bar{p}$.
\end{enumerate}
By the discussion above, $S$-valued points of $\Mbar_{1,1}([M/G],0)'$ can be identified as diagrams (\ref{object_in_Mwtilde}) without the section $p'$. Therefore forgetting $p'$ yields a morphism 
\begin{equation}\label{cover_with_section}
\phi: \Mwtilde_{1,1}([M/G], 0)'\to \Mbar_{1,1}([M/G], 0)'.
\end{equation}
By assumption the cover $D\to C$ is unramified over $\bar{p}$, so the group $G$ acts freely on the fiber of $D\to C$ over $\bar{p}$. Hence $G$ acts on $\Mwtilde_{1,1}([M/G], 0)'$ freely by permuting the choices of the section $p'$. The map (\ref{cover_with_section}) is $G$-equivariant with respect to the trivial $G$-action on $\Mbar_{1,1}([M/G], 0)'$, and it induces 
\begin{equation}\label{Mbar_vs_Mwtilde}
\Mwtilde_{1,1}([M/G], 0)'/G=\Mbar_{1,1}([M/G], 0)'.
\end{equation} 
In particular (\ref{cover_with_section}) is \'etale of degree $|G|$. 

Our next goal is to describe the moduli space $\Mwtilde_{1,1}([M/G], 0)'$. Consider an $S$-valued point (\ref{object_in_Mwtilde}) of $\Mwtilde_{1,1}([M/G], 0)'$. Let 
\begin{equation}\label{stein_factorization}
D\to \widetilde{S}\to S
\end{equation}
be the Stein factorization of $D\to S$. By definition of Stein factorization, the map $D\to \widetilde{S}$ has connected fibers, and the map $\widetilde{S}\to S$ is finite. Since $p': S\to D$ is a section of $D/S$, the composite $p'': S\overset{p'}{\longrightarrow}D\to \widetilde{S}$ is an isomorphism onto its image. The pull-back 
$$\begin{CD}
\widetilde{D}@> >> D\\
@V{}VV @V{}VV\\
S@>p''>>\widetilde{S}
\end{CD}$$
yields an $S$-family $\widetilde{D}\to S$ of {\em connected curves}. The natural map $\widetilde{D}\to D\to C$ is an $S$-family of connected admissible $H$-covers for some subgroup $H$ of $G$. The section $p': S\to D$ defines a section $\tilde{p}': S\to \widetilde{D}$. The composite $\widetilde{D}\to D\to M$ yields an $H$-equivariant map of degree $0$.

Composing the section $\tilde{p}': S\to \widetilde{D}$ with the degree $0$ map $\widetilde{D}\to M$ yields an $S$-valued point $S\to M$ of $M$. By $H$-equivariance, this is in fact an $S$-valued point $S\to M^H$ of the $H$-fixed locus $M^H$. 

Since $(\widetilde{D}, \tilde{p}')\to (C, \bar{p})$ is a connected pointed subcover of $(D, p')\to (C, \bar{p})$, the subgroup $H$ is the subgroup of $G$ generated by monodromies of the pointed cover $(D_t, p'(t)) \to (C_t, \bar{p}(t))$, where $t\in S$. If $C_t$ is smooth, then this is just the image of the natural map $\mathbb{Z}\oplus \mathbb{Z}\simeq \pi_1(C_t, \bar{p}(t))\to G$. If $C_t$ is singular, then by stability $C_t$ has one node. In this case the monodromy arises in two ways: from the fundamental group $\pi_1(C_t, \bar{p}(t))\simeq \mathbb{Z}$;  and from vanishing cycles (i.e. non-trivial elements in $\pi_1(C_t\setminus \{\text{node}\}, \bar{p}(t))$). The monodromy can also be understood as follows. One can see that $\widetilde{D}\to \C_t:=[\widetilde{D}/H]$ is a connected principal $H$-bundle over the {\em twisted curve} $\C_t$ whose coarse curve is $C_t$. The group $H$ is the image of the natural map $\pi_1^{orb}(\C_t, \bar{p}(t))\to G$. The twisted curve $\C_t$ has one possibly stacky node whose stabilizer group is of order $m\in \mathbb{N}$, hence its orbifold fundamental group $\pi_1^{orb}(\C_t, \bar{p}(t))$ is isomorphic to $\mathbb{Z}\oplus \mathbb{Z}_m$. Thus $H$ is also abelian. Therefore in either case $H$ is an abelian subgroup of $G$ with at most two generators. We call such a group {\em bicyclic}.

Conversely, suppose that $\widetilde{D}\to C$ is an $S$-family of connected admissible $H$-covers with a section of $\widetilde{D}/S$ for some bicyclic subgroup $H$ of $G$. Then $$D:=\widetilde{D}\times^{H}G$$ yields an $S$-family of admissible $G$-covers $D\to C$ with a section of $D/S$. Given an $H$-equivariant degree $0$ map $\widetilde{D}\to M^H$ this also yields a $G$-equivariant degree $0$ map $D\to M$. This yields an object of the form (\ref{object_in_Mwtilde}).

We next analyze automorphisms. An automorphism of the object (\ref{object_in_Mwtilde}) is a $G$-equivariant isomorphism $D\to D$ fixing $p'$ such that 
$$\begin{CD}
D@> >>M\\
@V{}VV @V{||}VV\\
D@> >>M
\end{CD}$$
is commutative. Since the map $D\to D$ must fix the section $p'$, we see that it induces an $H$-equivarinat isomorphism $\widetilde{D}\to\widetilde{D}$ fixing the section $\tilde{p}'$ such that 
$$\begin{CD}
\widetilde{D}@> >>M^H\\
@V{}VV @V{||}VV\\
\widetilde{D}@> >>M^H
\end{CD}$$
is commutative.

Conversely, an $H$-equivariant isomorphism $\widetilde{D}\to \widetilde{D}$ fixing $\tilde{p}'$ and commuting with $\widetilde{D}\to M^H$ uniquely yields a $G$-equivariant isomorphism $D\to D$ fixing $p'$ and commuting with $D\to M$. Here $D:=\widetilde{D}\times^{H}G$, $p'$ and the map $D\to M$ are constructed as above.

Let $\Mwtilde_{1,1}[H]^{conn}$ be the moduli stack whose $S$-valued points are 
\begin{equation}\label{object_in_Mconn}
(\widetilde{D}, \tilde{p}')\to (C, \bar{p})\to S,
\end{equation}
such that
\begin{enumerate}
\item
$(C, \bar{p})$ is a connected $1$-pointed genus $1$ stable curve;
\item
$\widetilde{D}\to C$ is an $S$-family of connected admissible $H$-covers unramified over $\bar{p}$; 
\item
$\tilde{p}': S\to \widetilde{D}$ is a section such that the composite $S\overset{\tilde{p}'}{\longrightarrow} \widetilde{D}\to C$ is the map $\bar{p}$.
\end{enumerate}
Then the discussion above proves the following 
\begin{proposition}\label{identification_of_covering_space}
There is an isomorphism of stacks
$$\Mwtilde_{1,1}([M/G], 0)'\simeq \coprod_{H\subset G \text{ bicyclic}}(\Mwtilde_{1,1}[H]^{conn}\times M^H).$$
\end{proposition}

\subsection{Obstruction theory}\label{section:obs_theory}
Consider the universal family 
$$\begin{CD}
\C@> f>> [M/G]\\
@V{\pi}VV\\
\Mbar_{1,1}([M/G], 0)'.
\end{CD}$$
Let $\mathfrak{M}_{1,1}^{tw}$ denote the stack of $1$-pointed genus $1$ twisted curves. It is known (see \cite{AGV}) that the stack $\Mbar_{1,1}([M/G], 0)'$ admits a perfect obstruction theory relative to the morphism $$\Mbar_{1,1}([M/G], 0)'\to \mathfrak{M}_{1,1}^{tw},$$ defined by forgetting the stable map. The obstruction theory is given by $R^\bullet\pi_*f^*T_{[M/G]}$. Since the map $\phi$ in (\ref{cover_with_section}) is \'etale, the pull-back $\phi^*R^\bullet\pi_*f^*T_{[M/G]}$ is a relative perfect obstruction theory on $\Mwtilde_{1,1}([M/G], 0)'$. Let $[\Mwtilde_{1,1}(\X,0)']^{vir}$ denote the associated virtual fundamental class. Then clearly we have $|G|[\Mbar_{1,1}(\X,0)']^{vir}=\phi_*[\Mwtilde_{1,1}(\X,0)']^{vir}$.

Let $\hat{\pi}:\widetilde{D}\to \Mwtilde_{1,1}[H]^{conn}$ denote the universal admissible $H$-cover over $\Mwtilde_{1,1}[H]^{conn}$. Then the following diagram
$$\begin{CD}
 \coprod_{H\subset G \text{ bicyclic}}(\widetilde{D}\times M^H)@> \tilde{f}>>M\\
@V{\tilde{\pi}=\hat{\pi}\times \text{id}}VV\\
\coprod_{H\subset G \text{ bicyclic}} (\Mwtilde_{1,1}[H]^{conn}\times M^H)
\end{CD}$$
is the universal family over $\coprod_{H\subset G \text{ bicyclic}} (\Mwtilde_{1,1}[H]^{conn}\times M^H)$. Here the map $\tilde{f}$ is obtained by projection to the second factor, together with the inclusion $M^H\subset M$. From the proof of the equivalence Proposition \ref{identification_of_covering_space} it follows that on the component $\Mwtilde_{1,1}[H]^{conn}\times M^H$ we have 
\begin{equation*}
\begin{split}
&\quad\phi^*R^1\pi_*f^*T_{[M/G]}|_{\Mwtilde_{1,1}[H]^{conn}\times M^H}\\
&=(R^1\tilde{\pi}_*\tilde{f}^*T_M)^H\\
&=(R^1\hat{\pi}_*\sO_{\widetilde{D}})^H\boxtimes T_{M^H}\\
&=\mathbb{E}^\vee\boxtimes T_{M^H}.
\end{split}
\end{equation*}
Here $\mathbb{E}$ is the pull-back of the Hodge bundle over $\Mbar_{1,1}$ via the natural map $$\Mwtilde_{1,1}[H]^{conn}\to \Mbar_{1,1}$$ which forgets the cover. It follows that the obstruction sheaf $\phi^*R^1\pi_*f^*T_{[M/G]}$ is locally free. 

\subsection{Computation}
In this Subsection we verify (\ref{dilaton})-(\ref{divisor}) by direct computations in the case $\X=[M/G]$, where we assume that $M$ is a smooth {\em projective} variety. Under this assumption the moduli space $\Mbar_{1,1}([M/G], 0)'$ is proper. Let $[\Mbar_{1,1}([M/G], 0)']^{vir}$ denote the virtual fundamental class associated to the obstruction theory discussed above.

Since the obstruction sheaf $\phi^*R^1\pi_*f^*T_{[M/G]}$ is locally free, it follows from e.g. \cite{ge}, Propostion 2.5, that the virtual fundamental class $[\Mwtilde_{1,1}([M/G],0)']^{vir}$ on $\Mwtilde_{1,1}([M/G], 0)'$ is given by 
\begin{equation}\label{vir_class_formula}
[\Mwtilde_{1,1}([M/G],0)']^{vir}=\bigoplus_{H\subset G \text{ bicyclic}} c_{top}(\mathbb{E}^\vee\boxtimes T_{M^H})\cap \left([\Mwtilde_{1,1}[H]^{conn}]\times [M^H]\right). 
\end{equation}

We now proceed to compute (\ref{genus1-inv}). First we have 
\begin{equation}\label{dilaton_calc_1}
\begin{split}
\<\tau_1(1)\>_{1,1,0}'^{\X}&=\int_{[\Mbar_{1,1}(\X,0)']^{vir}}\bpsi_1\\
&=\frac{1}{|G|}\int_{[\Mwtilde_{1,1}(\X,0)']^{vir}}\bpsi_1, \quad \text{ because $\phi$ is \'etale of degree } |G|\\
&=\frac{1}{|G|}\sum_{H\subset G \text{ bicyclic}}\int_{\Mwtilde_{1,1}[H]^{conn}}\bpsi_1\int_{M^H}c_{top}(T_{M^H})\quad \text{ by }(\ref{vir_class_formula}).
\end{split}
\end{equation}
We have abused notation by denoting the pull-back of $\psi_1\in H^2(\Mbar_{1,1}, \mathbb{Q})$ to {\em any} of these moduli spaces by $\bpsi_1$. 

Let $\Mbar_{1,1}[H]^{conn}$ be the moduli stack whose $S$-valued points are (\ref{object_in_Mconn}) without the section $\tilde{p}'$. Clearly, forgetting the section $\tilde{p}'$ yields a morphism $$\Mwtilde_{1,1}[H]^{conn}\to\Mbar_{1,1}[H]^{conn},$$ which is \'etale of degree $|H|$. Let $$\Mbar_{1,1}[H]^{conn}\to \Mbar_{1,1}$$ be the map defined by forgetting the admissible $H$-covers. The degree of this map is 
\begin{equation}\label{degree_forgetting_cover_map}
\frac{1}{|H|}\#\{g_1, g_2\in H| \<g_1, g_2\>=H, g_1g_2=g_2g_1\},
\end{equation}
see e.g. \cite{jk}. Therefore we have 
\begin{equation}\label{dilaton_calc_2}
\int_{\Mwtilde_{1,1}[H]^{conn}}\bpsi_1=|H|\int_{\Mbar_{1,1}[H]^{conn}}\bpsi_1=\#\{g_1, g_2\in H| \<g_1, g_2\>=H, g_1g_2=g_2g_1\}\int_{\Mbar_{1,1}}\psi_1.
\end{equation}
Notice that 
\begin{equation}\label{dilaton_calc_3}
\begin{split}
&\frac{1}{|G|}\sum_{H\subset G \text{ bicyclic}}\#\{g_1, g_2\in H| \<g_1, g_2\>=H, g_1g_2=g_2g_1\}\int_{M^H}c_{top}(T_{M^H})\\
=&\frac{1}{|G|}\int_{\coprod_{g_1, g_2\in G; g_1g_2=g_2g_1}M^{\<g_1,g_2\>}\times \{(g_1, g_2)\}}c_{top}(T_{M^{\<g_1,g_2\>}})\\
=&\int_{[(\coprod_{g_1, g_2\in G; g_1g_2=g_2g_1}M^{\<g_1,g_2\>}\times \{(g_1, g_2)\})/G]}c_{top}(T_{[(\coprod_{g_1, g_2\in G; g_1g_2=g_2g_1}M^{\<g_1,g_2\>}\times \{(g_1, g_2)\})/G]})\\
=&\int_{II[M/G]}c_{top}(T_{II[M/G]}),
\end{split}
\end{equation}
where in the last step we used the description of the double inertia stack as the following quotient
\begin{equation}\label{local_description_double_inertia}
II[M/G]=\left[\left(\coprod_{g_1, g_2\in G; g_1g_2=g_2g_1}M^{\<g_1,g_2\>}\times \{(g_1, g_2)\}\right)/G\right],
\end{equation}
Here the $G$-action on $\coprod_{g_1, g_2\in G; g_1g_2=g_2g_1}M^{\<g_1,g_2\>}\times \{(g_1, g_2)\}$ is given as follows: an element $g\in G$ sends a point $m\in M^{\<g_1,g_2\>}$ to $g\cdot m\in M^{\< gg_1g^{-1}, gg_2g^{-1}\>}$. This description can be found in e.g. \cite{toen}, the paragraph after the proof of Corollaire 3.46. 

Combining (\ref{dilaton_calc_1}), (\ref{dilaton_calc_2}), (\ref{dilaton_calc_3}), and the fact that $\int_{\Mbar_{1,1}}\psi_1=1/24$, we arrive at $$\<\tau_1(1)\>_{1,1,0}'^{\X}=\frac{1}{24}\int_{II[M/G]}c_{top}(T_{II[M/G]}),$$ which is (\ref{dilaton}) in this case.

Next we compute $\<\tau_0(D)\>_{1,1,0}'^{\X}$. 
\begin{equation}\label{divisor_calc_1}
\begin{split}
\<\tau_0(D)\>_{1,1,0}'^{\X}&=\int_{[\Mbar_{1,1}(\X,0)']^{vir}}ev_1^*D\\
&=\frac{1}{|G|}\int_{[\Mwtilde_{1,1}(\X,0)']^{vir}}ev_1^*D\\
&=\frac{1}{|G|}\sum_{H\subset G \text{ bicyclic}}\int_{\Mwtilde_{1,1}[H]^{conn}}(-\bpsi_1)\int_{M^H}D\cup c_{top-1}(T_{M^H})\quad \text{ by }(\ref{vir_class_formula}).
\end{split}
\end{equation}
Here we also abuse notation by denoting various pull-backs of $D\in H^2(\X,\mathbb{C})$ still by $D$. In the last equality of (\ref{divisor_calc_1}) we used the standard fact $c_1(\mathbb{E})=\psi_1$ on $\Mbar_{1,1}$. Similar to (\ref{dilaton_calc_3}) we have 
\begin{equation}\label{divisor_calc_3}
\begin{split}
&\frac{1}{|G|}\sum_{H\subset G \text{ bicyclic}}\#\{g_1, g_2\in H| \<g_1, g_2\>=H, g_1g_2=g_2g_1\}\int_{M^H}D\cup c_{top-1}(T_{M^H})\\
=&\int_{II[M/G]}D\cup c_{top-1}(T_{II[M/G]}).
\end{split}
\end{equation}
Combining (\ref{divisor_calc_1}), (\ref{dilaton_calc_2}), (\ref{divisor_calc_3}) and $\int_{\Mbar_{1,1}}\psi_1=1/24$, we arrive at $$\<\tau_0(D)\>_{1,1,0}'^{\X}=-\frac{1}{24}\int_{II[M/G]}D\cup c_{top-1}(T_{II[M/G]}),$$ which is (\ref{divisor}) in this case.

\subsection{Alternative formulation}\label{sec:alternative_approach}
In this Subsection we reformulate the ingredients used in the computation above. This reformulation is necessary for the study of the general case.

Let $\X=[M/G]$ be a global quotient stack with $G$ a finite group and $M$ a smooth {\em not necessarily projective} scheme. Consider the following composite
\begin{equation}\label{map_from_Mwtilde}
\coprod_{H\subset G \text{ bicyclic}}(\Mwtilde_{1,1}[H]^{conn}\times M^H)\to \coprod_{H\subset G \text{ bicyclic}}(\Mbar_{1,1}[H]^{conn}\times M^H)\to \Mbar_{1,1}\times \coprod_{H\subset G \text{ bicyclic}}M^H.
\end{equation}
Here the first map is given by the map $\Mwtilde_{1,1}[H]^{conn}\to\Mbar_{1,1}[H]^{conn}$ which forgets the section $\tilde{p}'$. This morphism is \'etale of degree $|H|$. The second map is given by the map $\Mbar_{1,1}[H]^{conn}\to \Mbar_{1,1}$, which forgets the covers. As mentioned above, its degree is (\ref{degree_forgetting_cover_map}). By Proposition \ref{identification_of_covering_space} this gives a morphism 
\begin{equation}\label{pre-pi_1}
\Mwtilde_{1,1}([M/G], 0)'\to \Mbar_{1,1}\times \coprod_{H\subset G \text{ bicyclic}}M^H.
\end{equation}
As mentioned above, there is a free $G$ action on $\Mwtilde_{1,1}([M/G], 0)'$ induced from the free $G$ action on the fiber over the marked point. The $G$-action on $M$ yields a $G$-action on $\coprod_{H\subset G \text{ bicyclic}}M^H$ as follows: an element $g\in G$ sends a point $m\in M^H$ to $g\cdot m\in M^{gHg^{-1}}$. It is straightforward to check that (\ref{pre-pi_1}) is $G$-equivariant with respect to these actions. By (\ref{Mbar_vs_Mwtilde}), this yields a morphism
\begin{equation}\label{pi_1}
\pi_1:\Mbar_{1,1}([M/G],0)'\to \Mbar_{1,1}\times \left[\left(\coprod_{H\subset G \text{ bicyclic}}M^H\right)/G\right],
\end{equation}
whose degree over each component is $\#\{g_1, g_2\in H| \<g_1, g_2\>=H, g_1g_2=g_2g_1\}$.

For a bicyclic subgroup $H\subset G$, there is an \'etale map $$\coprod_{g_1,g_2\in G, \<g_1, g_2\>=H, g_1g_2=g_2g_1} M^{\<g_1, g_2\>}\times \{(g_1, g_2)\}\to M^H,$$ which is of degree $\#\{g_1, g_2\in H| \<g_1, g_2\>=H, g_1g_2=g_2g_1\}$. Putting all these maps together gives a map $$\coprod_{g_1,g_2\in G, g_1g_2=g_2g_1} M^{\<g_1, g_2\>}\times \{(g_1, g_2)\}\to \coprod_{H\subset G\text{ bicyclic}}M^H,$$
which by (\ref{local_description_double_inertia}) induces a map 
\begin{equation}\label{pi_2}
\pi_2: \Mbar_{1,1}\times II[M/G] \to \Mbar_{1,1}\times \left[\left(\coprod_{H\subset G \text{ bicyclic}}M^H\right)/G\right].
\end{equation}
Here the $G$-action on $\coprod_{H\subset G \text{ bicyclic}}M^H$ is the same one used in (\ref{pi_1}).

Observe that by the discussion above, component-wise $\pi_1$ and $\pi_2$ have the same degree.

The discussion in Section \ref{section:obs_theory} shows that the obstruction sheaf on $\Mwtilde_{1,1}([M/G],0)'$ is the pull-back of the sheaf $\mathbb{E}^\vee\boxtimes T_{M^H}$ from $\Mbar_{1,1}\times \coprod_{H\subset G \text{ bicyclic}}M^H$ via (\ref{map_from_Mwtilde}). Hence the obstruction sheaf on $\Mbar_{1,1}([M/G],0)'$ is the pull-back of the sheaf $\mathbb{E}^\vee\boxtimes T_{[(\coprod_{H\subset G \text{ bicyclic}}M^H)/G]}$ from $\Mbar_{1,1}\times [(\coprod_{H\subset G \text{ bicyclic}}M^H)/G]$ via $\pi_1$. 

Suppose now that $M$ is projective. Then the moduli spaces are proper. We check (\ref{dilaton})-(\ref{divisor}) using this reformulation. By \cite{ge}, Propostion 2.5, we have the following equality of virtual fundamental classes,
\begin{equation*}
[\Mbar_{1,1}([M/G], 0)']^{vir}=c_{top}(\pi_1^*(\mathbb{E}^\vee\boxtimes T_{[(\coprod_{H\subset G \text{ bicyclic}}M^H)/G]}))\cap[\Mbar_{1,1}([M/G],0)'].
\end{equation*}
Applying $\pi_{1*}$ and using the fact that $\text{deg}\, \pi_1=\text{deg}\, \pi_2$, we obtain\footnote{Strictly speaking this computation is done component-wise.}
\begin{equation*}
\begin{split}
\pi_{1*}[\Mbar_{1,1}([M/G], 0)']^{vir}&=(\text{deg}\, \pi_1)c_{top}(\mathbb{E}^\vee\boxtimes T_{[(\coprod_{H\subset G \text{ bicyclic}}M^H)/G]})\cap \left([\Mbar_{1,1}]\times [[(\coprod_{H\subset G \text{ bicyclic}}M^H)/G]]\right)\\
&=(\text{deg}\, \pi_2)c_{top}(\mathbb{E}^\vee\boxtimes T_{[(\coprod_{H\subset G \text{ bicyclic}}M^H)/G]})\cap \left([\Mbar_{1,1}]\times [[(\coprod_{H\subset G \text{ bicyclic}}M^H)/G]]\right)\\
&=\pi_{2*}(c_{top}(\pi_2^*(\mathbb{E}^\vee\boxtimes T_{[(\coprod_{H\subset G \text{ bicyclic}}M^H)/G]}))\cap ([\Mbar_{1,1}]\times [II[M/G]]))\\
&=\pi_{2*}(c_{top}(\mathbb{E}^\vee\boxtimes T_{II[M/G]})\cap ([\Mbar_{1,1}]\times [II[M/G]])).
\end{split}
\end{equation*}

It follows that 
\begin{equation*}
\begin{split}
\<\tau_1(1)\>_{1,1,0}'^{\X}&=\int_{[\Mbar_{1,1}([M/G],0)']^{vir}}\bpsi_1\\
&=\int_{\pi_{1*}[\Mbar_{1,1}([M/G],0)']^{vir}}\bpsi_1\\
&=\int_{\pi_{2*}(c_{top}(\mathbb{E}^\vee\boxtimes T_{II[M/G]})\cap [\Mbar_{1,1}]\times [II[M/G]])}\bpsi_1\\
&=\int_{[\Mbar_{1,1}]\times [II[M/G]]}\bpsi_1\cup c_{top}(\mathbb{E}^\vee\boxtimes T_{II[M/G]}),
\end{split}
\end{equation*}
and 
\begin{equation*}
\begin{split}
\<\tau_0(D)\>_{1,1,0}'^{\X}&=\int_{[\Mbar_{1,1}([M/G],0)']^{vir}}ev_1^*D\\
&=\int_{\pi_{1*}[\Mbar_{1,1}([M/G],0)']^{vir}}ev_1^*D\\
&=\int_{\pi_{2*}(c_{top}(\mathbb{E}^\vee\boxtimes T_{II[M/G]})\cap [\Mbar_{1,1}]\times [II[M/G]])}ev_1^*D\\
&=\int_{[\Mbar_{1,1}]\times [II[M/G]]}ev_1^*D\cup c_{top}(\mathbb{E}^\vee\boxtimes T_{II[M/G]}),
\end{split}
\end{equation*}
from which (\ref{dilaton}) and (\ref{divisor}) follow easily.

\section{General case}\label{section:general_case}
In this Section we discuss the calculation for general $\X$. Let $\X$ be a smooth proper Deligne-Mumford stack with projective coarse moduli space $X$. We begin with a lemma.

\begin{lemma}\label{etale_atlas_of_Mbar}
Let $\Y\to \X$ be an \'etale map. Then the induced map $\Mbar_{1,1}(\Y, 0)'\to \Mbar_{1,1}(\X, 0)'$ is \'etale.
\end{lemma} 
\begin{proof}
We use the formal criterion for \'etaleness. Let $0\to I\to A\to B\to 0$ be a square zero extension. Consider a commutative diagram 
\begin{equation}\label{extension_diagram}
\begin{CD}
S_0=\text{Spec}\, B@> >>\Mbar_{1,1}(\Y, 0)'\\
@V{}VV @V{}VV\\
S=\text{Spec}\, A@> >>\Mbar_{1,1}(\X,0)'.
\end{CD}
\end{equation}
We need to prove the existence of a lifting $\text{Spec}\, A\to \Mbar_{1,1}(\Y, 0)'$. Diagram (\ref{extension_diagram}) is equivalent to the following commutative diagram 
\begin{equation*}
\xymatrix{
& & {\Y}\ar[d]\\
{\C_0}\ar[r]\ar[urr]\ar[d]&{\C}\ar[r]\ar[d] & {\X}\\
S_0\ar@/^/[u]^{\sigma_0}\ar[r]& S\ar@/^/[u]^{\sigma}.
}
\end{equation*}
We may rearrange this as 
\begin{equation*}
\xymatrix{
S_0\ar[r]^{\sigma_0}\ar[d]& {\C_0}\ar[d]\ar[r]& {\Y}\ar[d]\\
S\ar[r]^{\sigma}& {\C}\ar@{-->}[ur]\ar[r] &{\X}.
}
\end{equation*}
Since $\Y\to \X$ is \'etale, by formal criterion for \'etaleness there exists a unique lifting $\C\to \Y$. The family $(\C/S\to \Y, S\overset{\sigma}{\rightarrow}\C\to \Y)$ provides the needed lifting $\text{Spec}\, A\to \Mbar_{1,1}(\Y,0)'$.
\end{proof}

By \cite{av}, Lemma 2.2.3, we may find an \'etale cover $\coprod_i \X_i\to \X$ of $\X$ such that each $\X_i$ is of the form $\X_i=[M_i/G_i]$ for some smooth scheme $M_i$ and finite group $G_i$. Moreover, as can be seen from the proof of \cite{av}, Lemma 2.2.3, the groups $G_i$ are stabilizer groups of objects of the stack $\X$. Applying Lemma \ref{etale_atlas_of_Mbar} to this \'etale cover $\coprod_i \X_i\to \X$ we obtain an \'etale cover 
\begin{equation}\label{etale_cover_of_Mbar11}
\coprod_i \Mbar_{1,1}(\X_i,0)'\to \Mbar_{1,1}(\X,0)'
\end{equation}
of $\Mbar_{1,1}(\X,0)'$. It is easy to see that $\Mbar_{1,1}(\X_i,0)'\times_{\Mbar_{1,1}(\X,0)'}\Mbar_{1,1}(\X_j, 0)'\simeq \Mbar_{1,1}(\X_i\times_\X\X_j, 0)'$. Indeed if $((\C_i/S, p_i)\to \X_i, (\C_j/S, p_j)\to \X_j)$ is an object of $\Mbar_{1,1}(\X_i,0)'\times_{\Mbar_{1,1}(\X,0)'}\Mbar_{1,1}(\X_j, 0)'$, then $(\C_i/S, p_i)\to \X_i\to \X$ and $(\C_j/S,p_j)\to \X_j\to \X$ are isomorphic stable maps, which gives a stable map $(\C_i/S, p_i)\simeq (\C_j/S, p_j)\to \X_i\times_\X\X_j$. The converse is clear.

Let $\X_2$ denote the stack of pairs $(x, H)$ where $x\in Ob(\X)$ is an object of $x$ and $H\subset Aut(x)$ is a bicyclic subgroup of $Aut(x)$. Forgetting the bicyclic subgroups yields a natural map 
\begin{equation}\label{bicyclic_inertia_to_stack}
\pi_{\X_2}:\X_2\to \X.
\end{equation}
The collection $$\coprod_i [(\coprod_{H\subset G_i \text{ bicyclic}}M_i^H)/G_i]$$ forms an \'etale cover of $\X_2$.

The collection of maps $\pi_1$ in (\ref{pi_1}) for various $i$ induces a map 
\begin{equation}\label{pi_1_global}
\pi_1: \Mbar_{1,1}(\X, 0)'\to \Mbar_{1,1}\times \X_2.
\end{equation}
The collection of maps $\pi_2$ in (\ref{pi_2}) for various $i$ induces a map 
\begin{equation}\label{pi_2_global}
\pi_2: \Mbar_{1,1}\times II\X\to \Mbar_{1,1}\times \X_2.
\end{equation}
We observe that component-wise $\text{deg}\,\pi_1=\text{deg}\,\pi_2$ since they are the same on each \'etale chart, by the discussion in Section \ref{sec:alternative_approach}.

By the discussion in Section \ref{section:obs_theory}, the obstruction sheaf on each \'etale chart $\Mbar_{1,1}(\X_i, 0)'$ is the pull-back of the sheaf $\mathbb{E}^\vee\boxtimes T_{[(\coprod_{H\subset G_i \text{ bicyclic}}M^H)/G_i]}$ via $\pi_1$. It follows that the obstruction sheaf on $\Mbar_{1,1}(\X, 0)'$ is the pull-back of the sheaf $\mathbb{E}^\vee\boxtimes T_{\X_2}$. Also observe that $\pi_2^*(\mathbb{E}^\vee\boxtimes T_{\X_2})=\mathbb{E}^\vee\boxtimes T_{II\X}$ because this holds on each \'etale chart. 

Applying \cite{ge}, Propostion 2.5 gives the following equation for virtual fundamental classes:
\begin{equation}\label{virtual_class_formula_global}
[\Mbar_{1,1}(\X,0)']^{vir}=c_{top}(\pi_1^*(\mathbb{E}^\vee\boxtimes T_{\X_2}))\cap [\Mbar_{1,1}(\X,0)'].
\end{equation}

We can now proceed in a way similar to Section \ref{sec:alternative_approach}. 
\begin{equation*}
\begin{split}
\pi_{1*}[\Mbar_{1,1}(\X, 0)']^{vir}&=(\text{deg}\, \pi_1)c_{top}(\mathbb{E}^\vee\boxtimes T_{\X_2})\cap ([\Mbar_{1,1}]\times [\X_2])\\
&=(\text{deg}\, \pi_2)c_{top}(\mathbb{E}^\vee\boxtimes T_{\X_2})\cap ([\Mbar_{1,1}]\times [\X_2])\\
&=\pi_{2*}(c_{top}(\pi_2^*(\mathbb{E}^\vee\boxtimes T_{\X_2}))\cap ([\Mbar_{1,1}]\times [II\X]))\\
&=\pi_{2*}(c_{top}(\mathbb{E}^\vee\boxtimes T_{II\X})\cap ([\Mbar_{1,1}]\times [II\X])).
\end{split}
\end{equation*}

It follows that 
\begin{equation*}
\begin{split}
\<\tau_1(1)\>_{1,1,0}'^{\X}&=\int_{[\Mbar_{1,1}(\X,0)']^{vir}}\bpsi_1\\
&=\int_{\pi_{1*}[\Mbar_{1,1}(\X,0)']^{vir}}\bpsi_1\\
&=\int_{\pi_{2*}(c_{top}(\mathbb{E}^\vee\boxtimes T_{II\X})\cap [\Mbar_{1,1}]\times [II\X])}\bpsi_1\\
&=\int_{[\Mbar_{1,1}]\times [II\X]}\bpsi_1\cup c_{top}(\mathbb{E}^\vee\boxtimes T_{II\X}),
\end{split}
\end{equation*}
and 
\begin{equation*}
\begin{split}
\<\tau_0(D)\>_{1,1,0}'^{\X}&=\int_{[\Mbar_{1,1}(\X,0)']^{vir}}ev_1^*D\\
&=\int_{\pi_{1*}[\Mbar_{1,1}(\X,0)']^{vir}}ev_1^*D\\
&=\int_{\pi_{2*}(c_{top}(\mathbb{E}^\vee\boxtimes T_{II\X})\cap [\Mbar_{1,1}]\times [II\X])}ev_1^*D\\
&=\int_{[\Mbar_{1,1}]\times [II\X]}ev_1^*D\cup c_{top}(\mathbb{E}^\vee\boxtimes T_{II\X}),
\end{split}
\end{equation*}
from which (\ref{dilaton}) and (\ref{divisor}) follow easily.

\section{Localization approach}\label{section:localization_approach}
Let $T=(\com)^r$ be an algebraic torus. Assume that $\X$ admits a $T$-action with isolated fixed points. The virtual localization formula \cite{gp} expresses Gromov-Witten invariants of $\X$ as a sum of contributions from fixed loci. In this section we apply virtual localization\footnote{The hypothesis in the proof of localization formula in \cite{gp} is verified for moduli stacks of stable maps to Deligne-Mumford stacks by the work of \cite{agot}.} to verify (\ref{dilaton})-(\ref{divisor}).

\subsection{Fixed loci analysis}
Let $$[f: \C\to \X]\in \Mbar_{1,1}(\X,0)'^T$$ be a $T$-fixed stable map. Then the image $f(\C)$ must be a $T$-fixed point $p\simeq BG\in \X^T$. The locus $\Mbar_p\subset \Mbar_{1,1}(\X,0)'^T$ parametrizing $T$-fixed stable maps with images $p$ is thus identified with moduli stack of maps to $BG$,
$$\Mbar_p\simeq \Mbar_{1,1}(BG)'.$$
The fixed substack $\Mbar_{1,1}(\X,0)'$ is a union of $\Mbar_p$ over fixed points $p$,
$$\Mbar_{1,1}(\X,0)'=\cup_{p\in \X^T}\Mbar_p.$$
Clearly $\Mbar_p\simeq \Mbar_{1,1}(BG)'$ is smooth, and the virtual class induced from the $T$-fixed obstruction theory coincides with the fundamental class $[\Mbar_{1,1}(BG)']$. A dimension count shows that the virtual normal bundle has virtual rank $0$. Its Euler class may be described as follows. For a stable map $[f:\C\to \X]\in \Mbar_p$, let $\rho_f: \pi_1^{orb}(\C)\to G$ denote its {\em monodromy representation}. The vector space $H^0(\C, f^*T\X)$ of invariant sections is identified with the subspace $$(T_p\X)^{\rho_f}\subset T_p\X$$
which consists of vectors fixed by the monodromy representation $\rho_f$. Furthermore there is a $\rho_f$-equivariant splitting into fixed part and ``moving'' part $$T_p\X=(T_p\X)^{\rho_f}\oplus (T_p\X)^{mov}.$$ 
It follows that 
\begin{equation*}
\begin{split}
H^1(\C, f^*T\X)&\simeq H^0(\C, \omega_\C\otimes (f^*T\X)^\vee)^\vee\\
 &\simeq H^0(\C,\omega_\C\otimes (f^*(T_p\X)^{\rho_f})^\vee)^\vee\oplus H^0(\C, \omega_\C\otimes (f^*(T_p\X)^{mov})^\vee)^\vee\\
&=H^0(\C, \omega_\C)^\vee\otimes (T_p\X)^{\rho_f} \quad (\text{since } (T_p\X)^{mov} \text{ is not }\rho_f \text{ fixed})\\
&=\bE^\vee\otimes (T_p\X)^{\rho_f}.
\end{split}
\end{equation*}
Here $\bE$ is the pull back of the  Hodge bundle of $\Mbar_{1,1}$. As $[f]$ varies the spaces $(T_p\X)^{\rho_f}$ form a vector bundle $V$, which is trivial on components of $\Mbar_p$. Therefore the $T$-equivariant {\em inverse} Euler class of the virtual normal bundle is 
\begin{equation}\label{vir-normal}
\frac{c_{top}(\bE^\vee\otimes V)}{c_{top}(V)}=1-\bpsi_1\frac{c_{top-1}(V)}{c_{top}(V)}.
\end{equation}

\subsection{Localization on double inertia stack}
The $T$-action on $\X$ canonically lifts to a $T$-action on $II\X$, consequently the map $\pi_\X: II\X\to \X$ is $T$-equivariant. The $T$-fixed locus $II\X^T$ is decomposed according to the image under $\pi$:
$$II\X^T=\cup_{p\simeq BG\in \X^T}\pi_\X^{-1}(p), \quad \pi_\X^{-1}(p)\simeq IIBG.$$
Moreover, the normal bundle of $IIBG\subset II\X$ coincides with the restriction of $TII\X$. By the Atiyah-Bott localization formula, an integral over $II\X$ may be expressed as a sum of contributions from each $\pi_\X^{-1}(p)$.

\subsection{Dilaton}
(\ref{dilaton})-(\ref{divisor}) will be proven by equating contributions from $\Mbar_p$ to the left-hand side with contributions from $\pi_\X^{-1}(p)$ to the right-hand side.  

By (\ref{vir-normal}), the contribution to $\<\tau_1(1)\>_{1,1,0}'^{\X}$ from the fixed locus $\Mbar_p$ is 
\begin{equation}\label{dilaton-contribution}
\begin{split}
&\int_{\Mbar_{1,1}(BG)'}\bpsi_1\left(1-\bpsi_1\frac{c_{top-1}(V)}{c_{top}(V)}\right)\\
=&\int_{\Mbar_{1,1}(BG)'}\bpsi_1\\
=&\text{deg}(\Mbar_{1,1}(BG)'\to \Mbar_{1,1})\int_{\Mbar_{1,1}}\psi_1\\
=&\frac{1}{24}\frac{\#\{g,h\in G|gh=hg\}}{|G|},
\end{split}
\end{equation}
where the last equality uses the degree calculation in \cite{jk}, Proposition 2.1.

The contribution from $\pi_\X^{-1}(p)$ to the right-hand side of (\ref{dilaton}) is 
\begin{equation}\label{dilaton-rhs}
\frac{1}{24}\int_{\pi_\X^{-1}(p)}\frac{c_{top}(T_{II\X}|_{\pi_\X^{-1}(p)})}{c_{top}(T_{II\X}|_{\pi_\X^{-1}(p)})}
=\frac{1}{24}\int_{IIBG}1.
\end{equation}
In order to evaluation this integral, a description of the double inertia stack $IIBG$ is needed. 

Consider the subset 
$$A:=\{(g,h)\in G\times G|gh=hg\}\subset G\times G$$
of pairs of commuting elements of $G$. Let $G$ act on $A$ by simultaneous conjugation,
$$k\cdot (g,h):= (kgk^{-1}, khk^{-1}), \quad k\in G, (g,h)\in A.$$
An orbit of this $G$-action is called a {\em bi-conjugacy class} of $G$. For $(g,h)\in A$, the bi-conjugacy 
class containing $(g,h)$ will also be denoted by $(g,h)$.

For $(g,h)\in A$ let $C(g,h)\subset G$ denote the centralizer subgroup of $g, h$. If $(g,h)$ and $(g', h')$ belong to the same bi-conjugacy class, then $C(g,h)$ and $C(g',h')$ are conjugate to each other. In particular they have the same order. Clearly the $G$-action on $A$ has stabilizer at $(g,h)$ equal to $C(g,h)$. Thus the size of this bi-conjugacy class $(g,h)$ is $$\frac{|G|}{|C(g,h)|}.$$

It follows from the definition of double inertia stacks that $$IIBG=\coprod_{(g,h): \text{bi-conjugacy class}} BC(g,h).$$
This implies 
\begin{equation*}
\begin{split}
&\int_{IIBG}1=\sum_{(g,h): \text{bi-conjugacy class}}\int_{BC(g,h)}1\\
=&\sum_{(g,h): \text{bi-conjugacy class}}\frac{1}{|C(g,h)|}=\frac{1}{|G|}\sum_{(g,h): \text{bi-conjugacy class}}\frac{|G|}{|C(g,h)|}\\
=&\frac{|A|}{|G|}=\frac{\#\{g,h\in G|gh=hg\}}{|G|}.
\end{split}
\end{equation*}
This shows that (\ref{dilaton-rhs}) is equal to (\ref{dilaton-contribution}), as desired.

\subsection{Divisor}
An equivariant lift of the class $D$ needs be chosen. By abuse of notation, this lift will also be denoted by $D$. 

By (\ref{vir-normal}), the contribution to $\<\tau_0(D)\>_{1,1,0}'^{\X}$ from the fixed locus $\Mbar_p$ is 
\begin{equation}\label{divisor-contribution1}
\begin{split}
&\int_{\Mbar_{1,1}(BG)'}ev_1^*D\left(1-\bpsi_1\frac{c_{top-1}(V)}{c_{top}(V)}\right)\\
=&-\int_{\Mbar_{1,1}(BG)'}ev_1^*D\left(\bpsi_1\frac{c_{top-1}(V)}{c_{top}(V)}\right)\\
\end{split}
\end{equation}
(\ref{divisor-contribution1}) may be written as a sum of integrals over components of $\Mbar_{1,1}(BG)'$, as follows. Note that $\Mbar_{1,1}(BG)'$ is smooth, hence irreducible components do not intersect. Since the monodromy representation associated to a stable map is discrete, it is constant on irreducible components. Therefore irreducible components of $\Mbar_{1,1}(BG)'$ are indexed by orbits of the adjoint action of $G$ on the monodromy representations $\pi_1^{orb}(\C)\to G$, where $[\C]\in \Mbar_{1,1}$ is a generic point. To describe the orbits we may choose an identification\footnote{The particular indexing of the components depends on this choice.} $\pi_1^{orb}(\C)\simeq \mathbb{Z}\oplus \mathbb{Z}$. It follows that $\Mbar_{1,1}(BG)'$ is a disjoint union indexed by bi-conjugacy classes of $G$. Let $$\M(g,h)\subset \Mbar_{1,1}(BG)'$$ denote the component indexed by the bi-conjugacy class $(g,h)$. Given $[f:\C\to BG]\in \M(g,h)$, the image of the monodromy representation $\rho_f: \pi_1^{orb}(\C)\to G$ is the subgroup $\<g,h\>$ of $G$ generated by $g, h$. Thus the vector bundle $V|_{\M(g,h)}$ is trivial, with fiber $(T_p\X)^{\<g,h\>}$. Let $\lambda_i, 1\leq i\leq d(g,h)$ be its $T$-weights. Therefore (\ref{divisor-contribution1}) is equal to 
\begin{equation}\label{divisor-contribution2}
\begin{split}
&-\sum_{(g,h): \text{bi-conjugacy class of }G}\left(D|_p\sum_{1\leq i\leq d(g,h)}\frac{1}{\lambda_i}\right)\int_{\M(g,h)}\bpsi_1\\
=&-\sum_{(g,h): \text{bi-conjugacy class of }G}\left(D|_p\sum_{1\leq i\leq d(g,h)}\frac{1}{\lambda_i}\right)\cdot \text{deg}(\M(g,h)\to \Mbar_{1,1})\cdot \int_{\Mbar_{1,1}}\psi_1\\
=&-\frac{1}{24}\sum_{(g,h): \text{bi-conjugacy class of }G}\left(D|_p\sum_{1\leq i\leq d(g,h)}\frac{1}{\lambda_i}\right)\cdot \frac{1}{|C(g,h)|}.
\end{split}
\end{equation}
By the description of $IIBG$ given above, the contribution from $\pi_\X^{-1}(p)$ to the right-hand side of (\ref{dilaton}) is 
\begin{equation*}
\begin{split}
&-\frac{1}{24}\sum_{(g,h): \text{bi-conjugacy class of }G}\int_{BC(g,h)}D|_p\frac{c_{top-1}(T_{II\X}|_{BC(g,h)})}{c_{top}(T_{II\X}|_{BC(g,h)})}\\
=&-\frac{1}{24}\sum_{(g,h): \text{bi-conjugacy class of }G}\left(D|_p\sum_{1\leq i\leq d(g,h)}\frac{1}{\lambda_i}\right)\cdot \frac{1}{|C(g,h)|} \quad (\text{since }T_{II\X}|_{BC(g,h)}=(T_p\X)^{\<g,h\>}).
\end{split}
\end{equation*}
Clearly this agrees with (\ref{divisor-contribution2}).

\section{Twisted invariants}\label{section:twisted_invariants}
In this Section we discuss how the methods in previous sections can be applied to compute certain genus $1$ twisted orbifold Gromov-Witten invariants. Given the additional data of a complex vector bundle $F\to \X$ and an invertible multiplicative characteristic class $$\bc(\cdot):=\exp\left(\sum_{j\geq 0}s_j ch_j(\cdot)\right),$$
one can define the so-called {\em twisted orbifold Gromov-Witten invariants}. The construction and properties of twisted invariants were studied in \cite{t}. The following integral enters in the differential equations that fully determine twisted invariants in terms of usual Gromov-Witten invariants (see \cite{t}):
$$\int_{[\Mbar_{1,1}(\X,0)']^{vir}}(ev_1^*ch(F)Td^\vee(L_1))_{k+1} \bc(F_{1,1,0}).$$
Here $F_{1,1,0}$ is a $K$-theory class given by 
$$F_{1,1,0}:=R^\bullet\pi_*f^*F,$$
where the maps $\pi$ and $f$ appear in the universal family
$$\begin{CD}
\C@> f>> \X\\
@V{\pi}VV\\
\Mbar_{1,1}(\X, 0)'.
\end{CD}$$
Also, $L_1$ is the universal cotangent line bundle over $\Mbar_{1,1}(\X,0)'$ associated to the marked point. The symbol $ch(-)$ denotes the Chern character, $Td^\vee(-)$ denotes the dual Todd class, and $(-)_{k+1}$ indicates the degree-$2(k+1)$ component.

\begin{proposition}\label{tw-orb-gw}
The following equality holds in either equivariant\footnote{Here we do not assume that the torus action on $\X$ has isolated fixed points.} or non-equivariant Gromov-Witten theory:
\begin{equation}\label{twisted-inv}
\begin{split}
&\int_{[\Mbar_{1,1}(\X,0)']^{vir}}(ev_1^*ch(F)Td^\vee(L_1))_{k+1} \bc(F_{1,1,0})\\
=&-\frac{1}{48}\int_{II\X}ch_k(\pi_\X^*F)c_{top}(TII\X)+\frac{1}{24}\int_{II\X}ch_{k+1}(\pi_\X^*F)(\sum_j s_j ch_{j-1}((\pi_\X^*F)^{inv}))c_{top}(TII\X)\\
&+\frac{1}{24}\int_{II\X}ch_{k+1}(\pi_\X^*F)c_{top}(TII\X).
\end{split}
\end{equation}
\end{proposition}
Here $(\pi_\X^*F)^{inv}$ is the {\em invariant} subbundle of $\pi_\X^*F\to II\X$. A point in $II\X$ is by definition a pair of a point $p\simeq BG\in \X$ and a bi-conjugacy class $(g,h)$ of $G$.  The fiber of $(\pi_\X^*F)^{inv}$ over this point is the subspace $F|_p^{\<g,h\>}\subset F|_p$ invariant under the action of the group $\<g,h\>$.
\begin{proof}
First observe that if $\X$ admits an action by an algebraic torus $T$, then there are natural $T$-actions on $\X_2$ and $II\X$, making the map $\pi_{\X_2}: \X_2\to \X$ in (\ref{bicyclic_inertia_to_stack}) and $\pi_\X: II\X\to \X$ equivariant with respect to $T$-actions. Since the induced $T$-action on $\Mbar_{1,1}(\X,0)'$ is given by post-composing the $T$-action on $\X$ with stable maps (i.e. $T$ does not act on the domain of the stable maps), it follows that the maps $\pi_1$ in (\ref{pi_1_global}) and $\pi_2$ in (\ref{pi_2_global}) are $T$-equivariant. 

Consequently equation (\ref{virtual_class_formula_global}) is also valid $T$-equivarintly.

Working locally on each \'etale chart in (\ref{etale_cover_of_Mbar11}) and apply the description in Section \ref{section:obs_theory}, we find that $$F_{1,1,0}=\pi_1^*((\mathbb{C}-\mathbb{E}^\vee)\boxtimes (\pi_{\X_2}^* F)^{inv}).$$
Here $(\pi_{\X_2}^*F)^{inv}$ is the invariant sub-bundle of $\pi_{\X_2}^*F\to \X_2$. The fiber of $(\pi_{\X_2}^*F)^{inv}$ at a point $(x, H)\in \X_2$ is the subspace $F^H|_x\subset F|_x$ invariant under the action of $H\subset Aut(x)$. It is also easy to see that $$\pi_2^*((\mathbb{C}-\mathbb{E}^\vee)\boxtimes (\pi_{\X_2}^* F)^{inv})=(\mathbb{C}-\mathbb{E}^\vee)\boxtimes (\pi_{\X}^* F)^{inv}.$$

With these observations, the proposition follows easily by applying again the strategy used to verify (\ref{dilaton})-(\ref{divisor}) in previous sections.
\end{proof}

\begin{remark}
\hfill
\begin{enumerate}
\item
In non-equivariant Gromov-Witten theory (\ref{twisted-inv}) also follows from (\ref{dilaton})-(\ref{divisor}) and dimension consideration; see \cite{t}, Lemma 7.2.1.
\item
The occurrence of the invariant subbundle $(\pi_\X^*F)^{inv}$ may be seen via localization formula as follows. Let $p\simeq BG\in\X^T$ and $[f: \C\to \X]\in \Mbar_p$. Then $$F_{1,1,0}|_{[f]}=H^0(\C, f^*F)-H^1(\C, f^*F).$$ 
Over the component $\M(g,h)\subset \Mbar_p$ indexed by the bi-conjugacy class $(g,h)$ one finds 
\begin{equation*}
\begin{split}
&H^0(\C, f^*F)=F|_p^{\<g,h\>},\\
&H^1(\C, f^*F)\simeq H^0(\C, \omega_\C\otimes f^*F^\vee)^\vee \simeq H^0(\C, \omega_\C)^\vee\otimes F|_p^{\<g,h\>}.
\end{split}
\end{equation*}
\end{enumerate}
\end{remark}

\end{document}